\documentclass[12pt,a4paper,english]{smfart}
\usepackage[english]{babel}
\usepackage{smfthm}
\usepackage{amssymb}
\usepackage{amsmath}
\usepackage{amsthm}
\usepackage{amsfonts}
\usepackage{graphicx}
\usepackage{enumerate}

\usepackage{euscript,mathrsfs}
\usepackage{longtable}

\usepackage[T1]{fontenc}

\usepackage[all,cmtip]{xy}
\usepackage{alltt}

\usepackage{calrsfs}

\usepackage{soul}

\xyoption{all}

\usepackage{fullpage}

\usepackage{color}

\newcommand{\R}{\mathbb{R}}
\newcommand{\Q}{\mathbb{Q}}
\newcommand{\FF}{\mathbb{F}}
\newcommand{\Z}{\mathbb{Z}}
\newcommand{\NN}{\mathcal{N}}

\def\N{{\rm N}}

\newcommand{\C}{\mathbb{C}}

\def\disc{{\rm disc}}

\def\sign{{\rm sign}}

\def\m{{\mathfrak m}}

\def\f{{\mathfrak f}}

\def\ker{{\rm ker}}
\def\log{{\rm log}}
\def\Cl{{\rm Cl}}
\def\mod{{\rm mod}}

\def\Gal{{\rm Gal}}

\def\SS{{\overline S}}

\def\M{{\rm M}}
\def\J{{\mathcal J}}
\def\E{{\mathcal E}}
\def\U{{\mathcal U}}
\def\B{{\mathcal B}}

\def\AA{{\rm A}}
\def\F{{\rm F}}

\def\L{{\rm L}}

\def\K{{\rm K}}
\def\k{{\rm k}}

\def\W{{\rm W}}
\def\H{{\rm H}}

\date{\today}

\author{Roslan Ibara Ngiza Mfumu}
\address{Faculté des Sciences et Techniques,  Marien Ngouabi University, Brazzaville, Republic of Congo \and FEMTO-ST, Universit\'e de Franche-Comt\'e, CNRS, 15B Avenue des Montboucons, 25000 Besan\c con, France}
\email{ribarang@univ-fcomte.fr, roslancello7@gmail.com}

\author{Christian Maire}
\address{FEMTO-ST, Universit\'e de Franche-Comt\'e, CNRS, 15B Avenue des Montboucons, 25000 Besan\c con, France }
\email{christian.maire@univ-fcomte.fr}

\title{Genus theory, governing field, ramification and Frobenius}

\thanks{The authors thank the International Mathematical Union (IMU) and the GRAID program for their support. This work has been supported by the EIPHI Graduate School (contract ``ANR-17-EURE-0002``) and by the Bourgogne-Franche-Comt\'e Region; by the European Mathematical Society; by the REDGATE project funded by CNRS; and by the Doctoral School SPIM from Bourgogne-Franche-Comt\'e.}

\subjclass{11R37, 11R29, 11R45}

\keywords{Genus theory, governing field, Frobenius}

\begin{document}

\maketitle


\begin{abstract} 
In this work we develop, through a governing field,  genus theory for a number field $\K$ with tame ramification in $T$ and splitting in $S$, where $T$ and $S$ are finite disjoint sets of primes of $\K$. This approach extends that initiated by the second author in the case of the class group. It allows expressing the $S$-$T$ genus number of a cyclic extension $\L/\K$ of degree $p$ in terms of the rank of a matrix constructed from the Frobenius elements of the primes ramified in $\L/\K$, in the Galois group of the underlying governing extension. For quadratic extensions $\L/\Q$, the matrices in question are constructed from the Legendre symbols between the primes ramified in  $\L/\Q$ and the primes in $S$.
\end{abstract}


\section{Introduction}
Let $\K$ be a number field, and let $S$ and $T$ be two finite and disjoint sets of places of~$\K$. We assume that $T$ contains only finite places. Let $\K_T^S$ denote the maximal abelian extension of $\K$, totally decomposed at all places in $S$ (or even $S$-split), unramified outside of~$T$, and with at most tame ramification at the places $v\in T$ (or even $T$-tamely ramified). This is a finite extension, and the Artin map allows us to identify the Galois group $\Gal(\K_T^S/\K)$ with the $S$-ray class group of $\K$ modulo $\m:=\prod_{v \in T}v$, which we denote by $\Cl_{\K,\m}^S$. For more details, see Section \S \ref{section_corpsderayon}.

\smallskip

Now let $\L/\K$ be an extension of number fields ramified at the set $\Sigma$. The genus theory provides information about the class group $\Cl_{\L,\m_\L}^{S_\L}$ in terms of $\Sigma$ and the behavior of the $S$-units of $\L$ in $\L/\K$. See Theorem \ref{theo genres}. 

The first remarkable result in genus theory dates back to Gauss, concerning the $2$-Sylow subgroup of the class group of quadratic extensions of $\Q$ (see \cite[Chapter 1, \S 1]{Ko} and \cite[Chapter IV, \S 4, Exercise 4.2.10]{Gras-livre}). The phenomenon described by Gauss has been studied, developed, and generalized by many authors, including the works of Hasse \cite{Hasse}, Leopoldt \cite{Leopoldt}, Furuta \cite{Furuta}, and others. For more details, see \cite[III.4]{Gras-livre}.

The introduction of the sets $T$ and $S$ was initiated by Jaulent \cite{Jau}, Federer \cite{Federer}, and others. A very good overview of all this can be found in \cite[Chapter II, 2.4, Chapter III, 2.1]{Jau}.

\smallskip

The work presented here is inspired by \cite{Maire2}. We develop the $S$-$T$ genus theory  via a governing extension denoted by $\F_T^S/\K$, where the usual ramification conditions are interpreted through relations between Frobenius elements. As a consequence, and similar to \cite[Theorem 1.3]{Maire2}, questions in genus theory can be translated into questions about the behavior of Frobenius elements in a governing field, for which the Chebotarev density theorem becomes central.

\smallskip

When the base field $\K$ is given and the Galois group of $\L/\K$  is a fixed abelian group, Frei, Loughran, and Newton \cite{FLN} studied the asymptotic behavior of the genus number of $\L/\K$ (for the class group) with respect to the discriminant of $\L$. It would be interesting to revisit their results in light of our work.

\smallskip

Before presenting our results, let us begin by specifying the context.

\subsection{The context} \label{section_contexte}

\subsubsection{Ray class groups} \label{section_corpsderayon}
Let $\K$ be a number field, $T$ a finite set of finite places of $\K$, and $S$ a finite set of places of $\K$, disjoint from $T$. Let us denote $S = S_0 \cup S_{\infty}$, where $S_0$ contains only non-archimedean places and $S_\infty$ contains archimedean places, which we assume to be contained in the set $Pl_{\K,\infty}^{re}$ of real places of $\K$.

For a place $v$ of $\K$, let $\iota_v$ denote the embedding of $\K$ into its completion $\K_v$.

Set \begin{itemize}
 \item[$\bullet$] $I_{\K,T}$  the group of nonzero fractional ideals of $\K$ prime to $T$,
 \item[$\bullet$]  $\displaystyle{\m=\prod_{v \in T} v}$ the ray modulus of $\K$ associated to  $T$,
\item[$\bullet$] $P_{\K,\m}^{S_{\infty}}$ the subgroup of principal ideals $(x)$ of $I_{\K,T}$, $x\equiv 1 \,(\m)$, and $\iota_v(x)>0$ for all $v\in pl^{re}_{\K,\infty}\backslash S_{\infty}$, 
\item[$\bullet$]  $\langle S_0\rangle$ the subgroup of  $I_{\K,T}$ builted over places of $S_0$, 
 \item[$\bullet$]
$R_{\K,\m}^S$ the subgroup $P_{\K,\m}^{S_{\infty}}\langle S_0\rangle$ of $I_{\K,T}$.
\end{itemize}

Let  $\Cl_{\K,\m}^S$ be the $S$-ray-class group modulo  $\m$, {\it i.e.}   $$\Cl_{\K,\m}^S:=I_{\K,T}/R_{\K,\m}^S.$$  By class field theory,  $\Cl_{\K,\m}^S$ is isomorphic to the Galois group of  $\K_{T}^S/\K$, where $\K_{T}^S/K$ is the maximal abelian extension $\K$, $T$-tamely ramified and $S$-split, see 
  \cite[Chapter II, \S 5]{Gras-livre}.

\subsubsection{Genus fields and genus numbers}
 Let $p$ be a prime number and let $\L/\K$ be a cyclic extension of degree $p$. Denote by $\Sigma$ the set of ramification of $\L/\K$; for an infinite place $v$, we will speak of decomposition versus non-decomposition, but not of ramification.

\smallskip

Let $T_\L$ (respectively $S_\L$) denote the  places of $\L$ lying  above those of $T$ (resp. of $S$), and consider  $\Cl_{\L,\m_\L}^S$, where $\m_\L:=\prod_{w \in T_\L} w$. 

\smallskip

Let $\M/\K$ be the maximal abelian extension  $\K$ contained in $\L_T^S/\K$. 

$$
\xymatrix{
\L \ar@{-}[d] \ar@{-}[r]& \ar@{-}[d]\ar@{-}[r] & \M \ar@{-}[r] & \L^S_{T} \\
 \K\ar@{-}[r] &\K^S_{T} \\
 }
$$ 

The field $\M$ is the     {\it $S$-$T$ genus field} associated with $\L/\K$, and the quantity $(g_T^S)^*=[\M:\L]$ is the {\it  $S$-$T$ genus number}. Set $g^S_T=[\M:\K_T^S]$: this is the quantity $g^S_T$ that we are studying. It can be observed that it is easy to pass from $g_T^S$ to $(g_T^S)^*$ as soon as $\# \Cl_{\K,\m}^S$ is known, and the knowledge of  $g_T^S$ provides information about $\Cl_{\L,\m_\L}^S$.
Of course, the genus theory makes sense when the field  $\L$ is not contained in $\K_{T}^S$, because otherwise  $\M=\K_T^S$. 
The extension $\M/\K$ being abelian, its Galois group can be approached through class field theory, which allows expressing $[\M:\K]$ in terms of the ramification in $\L/\K$ and the units of $\K$, thus leading to a non-trivial lower bound for   $\# \Cl_{\L,\m_\L}^{S_\L}$.

\smallskip

Since $\L/\K$ is cyclic of degree $p$, the Galois group $\Gal(\M/\K_T^S)$ is abelian of exponent $p$ (see Theorem \ref{theo genres}). 
Thus, when $p>2$, the infinite places play no role. Consequently, we assume that $S_\infty=Pl_{\K,\infty}^{re}$ in this case.

\subsubsection{Governing fields} 

We continue with a fixed prime number $p$. We then assume that for  $v\in T$, we have  $N_v =1 $ mod $p$, where $N_v$  is the cardinality of the residue field of the completion $\K_v$ of $\K$ at $v$. Note that without this condition, the $p$-part contributed by the place $v$ of $T$ in $\Cl_{\K,\m}^S$, would be trivial.

\smallskip

Let  $E_{T}^{S} $ be the group of  $S$-units of $\K$ congruent to  $1\,(\mod \, \m)$, that is,  $$E_{T}^{S}  =  \{x\in \K^\times\,\, |\,\, x\equiv 1\,(\mod \, \m) \,,\,v(x)=0 \,\,\forall v\not\in S\}.$$
Here, let us clarify the meaning  $v(x) = 0$.  If $v\in S_0$, we identify the  place $v$ with its  valuation; if $v \in Pl_{\K,\infty}^{re}$, $v(x)=0$ means  $\iota_v(x)>0$; for $v$ archimedean, $v\notin Pl_{\K,\infty}^{re}$, we always have $v(x)=0$.
 
 \medskip
 
 Let $\K'=\K(\zeta_p)$, where  $\zeta_p$ is a primitive $p$th root of unity. 

 {\it The governing field  $\F_T^S$}  associated with the triplet $(\K, T,S)$ is defined as  $$\F_T^S:=\K'(\sqrt[p]{E_T^S}).$$ 
 
We then define $\Gamma_T^S:=\Gal(\F_T^S/\K')$: it is an abelian  $p$-elementary group. 
 
 \medskip
 
 When $T=\emptyset$ and $S=Pl_{\K,\infty}^{re}$, by Dirichlet's theorem the $p$-rank of  $\Gamma^S:=\Gamma_\emptyset^S$ is $r_1+r_2-1+\delta_{\K,p}+\#S_0$, where $(r_1,r_2)$ is the signature of $\K$, and where  $\delta_{\K,p}=1$ if $\zeta_p \in \K$, and $0$ otherwise.

\subsection{Our result} \label{section_resultat}

To simplify the presentation of our result, we assume that the set $\Sigma$ does not contain any places above $p$; in other words, the extension $\L/\K$ is tamely ramified. To further simplify the presentation, we also assume that the places in $S$ split in $\L/\K$.

\medskip

For each place $v\in \Sigma$, we choose a place $w$ of $\K'$ above $v$  and set $\sigma_v:=\sigma_w$, the Frobenius element associated with $w$ in $\Gamma_T^S:=\Gal(\F_T^S/\K)$; of course, this element depends on the choice of $w$, but we will see that the conditions involving it are independent of this choice.

\medskip

\smallskip

Let $m=\# \Sigma$, and let $\{e_{v_1},\ldots , e_{v_{m}}\}$ be a basis of $(\FF_p)^{m}$ indexed by the places  $v$ of $\Sigma$.

We then consider the linear map  $\Theta_{\Sigma,T}^S$ defined by
$$
\begin{array}{rcll}
\Theta_{\Sigma,T}^S \, : \, (\FF_p)^{m}& \longrightarrow  & \Gamma_T^S \\
  e_{v}  & \longmapsto  & \sigma_{v} .
\end{array}
$$

We have the following result

\begin{theo}\label{Theo1}
Under the previous conditions, we have
    $$g^S_T=\#\ker(\Theta_{\Sigma,T}^S).$$
\end{theo}

\medskip

The essence of our work is to translate the ramification conditions through Frobenius elements in a governing field. Therefore, if we ensure that the Frobenius elements associated with the places of $T$ form a linearly independent set in $\Gamma^S:=\Gal(\F^S/\K')$, then we can express quite easily the Galois group $\Gamma_T^S$.

\smallskip
  Set $H_T:=\sum_{v\in T} \FF_p \sigma_v \subset \Gamma^S$. 

\begin{prop} \label{prop_T}
Suppose that the family $\lbrace \sigma_v,\,v\in T \rbrace$ forms a linear independent set over  $\FF_p$ in $\Gamma^S$.
Then $\Gamma_T^S\simeq \Gamma^S/H_T$.
\end{prop}

The condition of linear independence has an interpretation. Indeed, according to the Gras-Munnier theorem (see \cite[Chapter V, Corollary 2.4.2]{Gras-livre}, \cite{GM}), a relation between the Frobenius elements
$\sigma_v$, $v \in T$, implies the existence of a cyclic extension of degree $p$ of $\K$, $T$-ramified and $S$-split, and consequently contributes "trivially" to $g_T^S$. Thus, the condition of linear independence forces the context to avoid this situation.

\medskip

Theorem \ref{Theo1} becomes interesting when we have a good understanding of the governing field $\F^S$, especially when we know about the units. Typically, this occurs for $\K=\Q$, but also, as noted in \cite[\S 3.5.3]{Maire2}, for $p=3$ and for base field $\K=\Q(\zeta_3)$.

By introducing $S$-places, the units become $S$-units, and when the field $\K$ is principal, the governing field is relatively  easy to describe. A remarkable situation arises when $p=2$ and $\L/\Q$ is a quadratic extension. The quantity  $g_T^S$ corresponds to the kernel of a matrix constructed using Legendre symbols. Let's develop a very specific situation.

Let $\Sigma=\{p_1,\cdots, p_m\}$ and $S_0=\{\ell_1,\cdots, \ell_{s'}\}$. We assume that  $\Sigma\cap\left(\{S_0\}\cup \{2\}\right)=\emptyset$.
When  $S_\infty$ contains the unique infinite  place $v_\infty$, let' set $\ell_0=-1$ (we identify $\ell_0$ with $v_\infty$). 
Set  $s:=\#S$.

Here, we don't impose any condition on the behavior of  $v\in S$ in $\L/\Q$.

\smallskip

Let $A=(a_{i,j})$ be the matrix of size $s \times m$ defined by $$a_{i,j}=\left(\dfrac{\ell_j}{p_i}\right),$$ where $\left(\dfrac{\ell_i}{p_j}\right) \in \FF_2$ is the additive  Legendre symbol.

Next, consider  the diagonal matrix $D=(d_{i,j})$  size $s \times s$ defined by
$$d_{j,j}= \left\{\begin{array}{l} 1 \,\, {\rm if } \ \ell_j \ \text{\rm is inert in }\L/\Q, \\
      0                          \,  \, {\rm if }  \ \ell_j \ \text{\rm splits in }\L/\Q.                                                                                \end{array}
 \right.$$ 

Finally, let  $M=\left( A D \right)  $ be the matrix of  size $(m+s)\times s$. 

\begin{coro}
 Under the previous conditions, we have:
    $$g^S_\emptyset=\#\ker(A).$$
\end{coro}

\

The rest of our work consists of four sections. In Section \ref{section_genre}, we introduce and develop the elements of genus theory that are useful for ours results. Section \ref{section_gouvernant} is dedicated to the governing field. It is also in this section that we prove Proposition \ref{prop_T}. Section \ref{section_resultatprincipal} focuses on the main theorem and its proof. In the final section, we provide two applications.


\section{Elements of genus theory} \label{section_genre}

\subsection{$S$-$T$ genus formula} 

For this part, we refer, for example, to \cite[Chapter IV, \S 4]{Gras-livre}, \cite[Chapter III, \S 2]{Jau}, or \cite{Maire1}.

Let's go back to the framework of Section \S \ref{section_contexte}.
Let $T$ and $S$ be two finite disjoint sets of places of   $\K$, non-archimedean for $T$ and arbitrary for   $S=S_0\cup S_\infty$. 

Let $\L/\K$ be a cylic extension of degree $p$.

We denote by $$ E^S _T\cap \NN_{\L/\K}:=E^S_T\cap\N_{\L/\K}(U^S _{\L,\m}), $$ the elements of  $E^S _T$ that are locally norms everywhere in  $\L/\K$.

\medskip

The following theorem can be formulated in a more general context  (see \cite[Chapter IV]{Gras-livre}), but we will focus on the case of cyclic extension of degree $p$. 

\begin{theo}\label{theo genres} 
Let $\L/\K$ be a cylicc extension of degree $p$ with ramification set $\Sigma$. 
Then  $\Gal(\M/\K^S_{T})$ is an abelian group of exponent  $p$. In particular, $g_T^S$ is a power of  $p$, and 
$$\log_p\left(g^S_T\right)= \# S^{ns}+\#\Sigma \backslash  \Sigma \cap (S\cup T)- \log_p\left(E^S _T:E^S_T\cap\NN_{\L/\K}\right),$$
where $S^{ns}$ denotes the set of places of places in   $S$ that are not split in  $\L/\K$.
\end{theo}

Thus, the study of $g^S_T$ is closely related to the quantity $ E^S _T\cap \NN_{\L/\K}$. This will be done through the governing field $\F_T^S$. To achieve this, Proposition \ref{proposition1} from the upcoming section is central.

\medskip

We set $\Sigma':=\Sigma \backslash  \Sigma \cap (S\cup T)$.


\subsection{Genus fields and ray class fields}

Let  $\L/\K$ be an abelian extension. 
For  $v\in Pl_\K$, we denote by  $D_v:=D_v(\L/\K)$ the decomposition group of  $v$ in $\L/\K$ and by   $I_v:=I_v(\L/\K)$ its inertia group. It is worth mentioning that for an archimedean place $v$, we do not speak of ramification but rather of non-decomposition. 

We now make the choice of a place  $w|v$, and we set $\L_v:=\L_w$. Recall that the local reciprocity symbol provides the following isomorphisms:
$$\dfrac{\K_v^{\times}}{\N_{\L_v/\K_v}\L_v^{\times}}\simeq D_v, \ \ \dfrac{U_v}{\N_{\L_v/\K_v}U_{\L_v}} \simeq I_v\cdot$$
Here,   $U_v \subset \K_v^\times$ denotes the group of local units at $v$, and  $\N_{\L_v/\K_v}$ denotes the norm map of the local extension $\L_v/\K_v$.
For a real infinite place $v$, we adopt the convention $U_v=(\R^\times)^2$, and for a complex place $U_v=\C^\times$.

\smallskip

Thus, \begin{itemize}
\item[$-$] for places $v\in \Sigma':=\Sigma \backslash \Sigma \cap (S\cup T)$, the Artin map induces a surjective morphism   from $U_v$ to $I_v$, with kernel $W_v:=\N_{\L_v/\K_v}U_{\L_v}$,
\item[$-$] for places  $v\in S$, the Artin map induces a surjective morphism   from  $\K_v^\times$ to  $D_v$, with kernel $W_v:=\N_{\L_v/\K_v} \L_v^\times$. Note that $W_v=\K_v^\times$ if and only if $v$ splits in $\L/\K$.
\end{itemize}

\medskip

Set $$\displaystyle{W = \prod_{v\in (S\cup \Sigma) \backslash (T\cap \Sigma)}W_v= \prod_{v\in \Sigma'}W_v \prod_{v\in S} W_v}.$$

\medskip

\begin{rema}\label{rema0}
Observe that if $\L/\K$ s cyclic of degree  $p$, then for any place $v\in \Sigma'\cup S$ we have $\iota_v((E_T^{S})^p) \subset W_v$.
\end{rema}

\medskip

\begin{defi}
We denote by $\K_{\Sigma, S,T}$ the abelian extension of $\K$ corresponding, via the global Artin map, to the id\`ele group  $V$: $$V:=W\left( \prod_{v\notin \Sigma' \cup S \cup T } U_v \right) \left( \prod_{v\in T} U_v^1 \right)  = \left(\prod_{v\in \Sigma' \cup S }W_v \right) \left(\prod_{v\notin  \Sigma' \cup S \cup T} U_v\right) \left( \prod_{v\in T} U_v^1\right).$$
Here, $U_v^1$ is the subgroup of principal units of $U_v$. 
\end{defi}

The following proposition is central.

\begin{prop}\label{proposition1}  

We have $\M=\K_{\Sigma, S,T}$. Moreover

$$\Gal(\K_{\Sigma,S,T} /\K^S_{T})\simeq \dfrac{U_{\K, \Sigma'}^S}{\nu(E^S _T) W},$$
where  $U_{\K,\Sigma'}^S=\prod_{v\in S}\K_v^\times \prod_{v\in \Sigma'}U_v$, and where $\nu\,:\, E^S_T \longrightarrow U_{\K,\Sigma'}^S$ is the diagonal embedding.
\end{prop}

\begin{proof}
   Let's note that: 
   \begin{itemize} \item[$-$] a  {\it finite} place $v \notin \Sigma'\cup S\cup T$ of $\K$ is unramified in $\M/\K$,
    \item[$-$] a place $v\in T$ is tamely ramified in $\M/\K$.
   \end{itemize}

   Therefore, the global Artin symbol  for the extension $\M/\K$ is trivial on $$\left( \prod_{v\notin \Sigma'\cup S \cup T} U_v \right) \left( \prod_{v\in T} U_v^1 \right).$$ Now let's look at the part  $W$.
   
   For $v\in S$, since  $v$ splits totally in  $\L_S^T/\L$ and thus in $\M/\L$, then $\M_v=\L_v$. Consequently,  every element $\varepsilon$ of  $W_v$ is also locally norm in $\M/\K$. In other words, the local symbol at~$v$ in the extension $\M/\K$ vanishes on $W_v$.
   
   \smallskip
   
   For $v \in \Sigma'$, let  $\varepsilon \in W_v$. Then, by definition of $W_v$, there exists $z\in U_v$ such that $\varepsilon = N_{\L_v/\K_v} (z)$. But since the extension $\M_v/\K_v$ is unramified at  $v$, the element  $z$ is norm in  $\M_v/\L_v$, and thus $\varepsilon$ is norm in  $\M_v/\K_v$.  In other words, here too, the local symbol at $v$ in the extension $\M/\K$ vanishes on $W_v$.

\smallskip In conclusion, the global Artin symbol for the extension $\M/\K$ is trivial on $V$. Therefore, by maximality of $\K_{\Sigma,S,T}$, we have   $\M\subset\K_{\Sigma,S,T}$.

\medskip

Let's show the reverse inclusion. For that, observe that
 $\K_{\Sigma, S,T}/\L$ is an abelian extension such that:  
\begin{itemize}
    \item[$-$] every place $v\in T$ is tamely ramified (possibly unramified);
    \item[$-$] for every place $v\in S$, the following commutative diagram holds: 
$$\xymatrix{\K_v^\times/W_v \ar@{->>}[r] \ar@{->}^\simeq[dr]&D_v(\K_{\Sigma, S,T}/\K)\ar@{->>}[d] \\ 
&D_v(\L/\K) }$$
showing that  $D_v(\K_{\Sigma, S,T}/\L)$ is trivial, hence  $\K_{\Sigma, S, T}/\L$ is decomposed at every place $v\in S$;
    \item[$-$] similary, every place $v\in \Sigma'$ is unramified in $\K_{\Sigma, S,T}/\L$.
\end{itemize}
Thus, $\K_{\Sigma,S,T} $ is contained in $\L_S^T$, and by maximality of  $\M$, we deduce that  $\K_{\Sigma, S,T}\subset \M$. Consequently, $\M=\K_{\Sigma, S,T}$.

\medskip

In summary, if we denote by    $\J_\K$ the id\`ele group of $\K$, and by  $\U_{\K,T}^S$ the id\`ele subgroup given by $$\U_{\K,T}^S:= \prod_{v\in S} \K_v^\times \prod_{v\in T}U_v^1 \prod_{v\notin T\cup S} U_v,$$ we have 
$$\Gal(\K_{\Sigma, S,T}/\K)\simeq \J_\K/ V \K^\times  \text{ and } \Gal(\K^S_{T} /\K)\simeq \J_\K/ \U^S _{\K,T}\K^\times .$$ 

Therefore,  $$\Gal(\K_{\Sigma,S,T}/\K_T^S)\simeq  \U^S_{\K,T}\K^\times/ V \K^\times \simeq  \U^S_{\K,T} /\left(V \K^\times\right)\cap  \U^S_{\K,T} \simeq \U^S_{\K,T}/ V E_T^S.$$
We conclude by noticing that
$\U_{\K,T}^S/V \simeq U_{\K,\Sigma'}^S/ W$.
\end{proof}


\section{Governing fields} \label{section_gouvernant}

Set $\K'=\K(\mu_p)$. We fix a generator $\zeta_p$ of $\mu_p$.
If $B$ is an $\FF_p$-module, let $B^{\lor}:=\hom(B,\mu_p)$.

By Kummer duality, recall that  for a subgroup of  $A$ of $\K'^\times$, one has $A(\K'^\times)^p/(\K'^\times)^p \simeq  \Gal(\K'(\sqrt[p]{A})/\K')^{\lor}$. 
Moreover, if $A\subset \K^\times$, then 
$$\Gal(\K'(\sqrt[p]{A})/\K')^{\lor} \simeq A(\K'^\times)^p/(\K'^\times)^p \simeq A/A\cap (\K'^\times)^p \simeq  A/A\cap (\K^\times)^p \simeq 
 A(\K^\times)^p/(\K^\times)^p,$$
 because $[\K':\K]$ is coprime to $p$.
 
\subsection{Frobenius} \label{section_frob}

For any place $v$ of $\K$, let's define $$\E^S_{T,v}=\lbrace \varepsilon \in E_T^S,\,\,\varepsilon\in (\K_v^\times)^p \rbrace.$$ This group of $S$-units fits into the exact sequence 
$$1\longrightarrow \E^S_{T,v}(\K^\times)^p/(\K^\times)^p \longrightarrow E_T^S(\K^\times)^p/(\K^\times)^p \longrightarrow i_v
(E_T^S )\longrightarrow 1,$$ where  $i_v\,:\, E_T^S \longrightarrow \K_v^\times/(\K_v^\times)^p$ is induced by the embedding   $\iota_v$  of $\K$ into $\K_v$.

By Kummer duality we have 
$$i_v(E_T^S )^{\lor}\simeq (E_T^S(\K^\times)^p/\E^S_{T,v} (\K^\times)^p)^{\lor}\simeq \Gal(\K'(\sqrt[p]{E_T^S})/\K'(\sqrt[p]{\E^S_{T,v}})).$$
This latter Galois group is easily to interpret:
\begin{lemm}\label{lemm_dec}
 One has $\Gal(\K'(\sqrt[p]{E_T^S})/\K'(\sqrt[p]{\E^S_{T,v}}))=D_v(\F_T^S/\K')$.
\end{lemm}

(We will see latter that it does not depend on the choice of a place $w|v$ of $\K'$.)

\begin{proof} Let's denote by $\N$ the subfield of  $\F_T^S/\K'$ corresponding, via Galois theory, to  $D_v(\F_T^S/\K')$. Clearly, $\K'(\sqrt[p]{\E^S_{T,v}}) \subset \N$. For the reverse inclusion,  note that if there exists an intermediate subfield $\N'$ of degree $p$ over $\N$, then, as $\Gal(\K'(\sqrt[p]{E_T^S})/\K'(\sqrt[p]{\E^S_{T,v}})$ is an abelian $p$-elementary group,  $\N'$ arises from the compositum with    a cyclic extension $\N_0/\K'$ of degree $p$: there exists $x \in  E_T^S$ such that $\N_0=\K'(\sqrt[p]{x})$. Now, since $v$ splits  in $\N/\K'$,  it follows that $x\in (\K'_v)^p$, hence  $x\in \K_v^p$ because $[\K'_v:\K_v]$ is coprime to $p$; thus $\N_0 \subset \K'(\sqrt[p]{\E^S_{T,v}})$, which leads to a contradiction.
\end{proof}

When $v$ is unramified in $\F_T^S/\K'$, the Galois group of  $\K'(\sqrt[p]{E_T^S})/\K'(\sqrt[p]{\E^S_{T,v}})$ is generated by the Frobenius element associated to the choice of a place $w|v$ of $\K'$.

\medskip

For now on, let'us fix $w|v$ and set  $\sigma_v:=\sigma_w$, where  $\sigma_w$ is the Frobenius at  $w$ in $\Gal(\K'(\sqrt[p]{\E^S})/\K'$.

\medskip

Next, let  $D_v$ be the  decomposition group of $v$ in the extension $\F_T^S/\K'$. 

\medskip

Let
$$\Phi_v\,:\, \left( E_T^S (\K^\times)^p/\E^S_{T,v} (\K^\times)^p\right)^{\lor} \longrightarrow \Gal(\F_T^S/\K'(\sqrt[p]{\E^S_{T,v}}))=D_v$$
be the isomorphism arising from Kummer duality.
Recall how $\Phi_v$ is defined: for $\chi\in\left(E_T^S(\K^\times)^p/\E^S_{T,v} (\K^\times)^p  \right)^{\lor}$, we associate the element $g_\chi := \Phi_v(\chi)$ defined as follows:
$$g_\chi(\sqrt[p]{\varepsilon})=\chi(\varepsilon)\cdot \sqrt[p]{\varepsilon},$$
for any $\varepsilon \in E_T^S$.

\smallskip

For $v\in \Sigma'\cup S$, consider the local map $\varphi_v$  also derived from Kummer duality:
$$\varphi_v\,:\, \left( A_v/W_v \right)^{\lor} \hookrightarrow \left( A_v/A_v^p \right)^{\lor} \twoheadrightarrow i_v(E_T^S )^{\lor} \stackrel{\simeq}{\longrightarrow}  D_v,$$
where $A_v=U_v$ (respectively $A_v=\K_v^\times$) for $v\in \Sigma'$ (resp. $v\in S$).

When $\left( A_v/W_v \right)^{\lor}$ is non-trivial, it is generated by a certain character $\chi_v=\chi_w$.
Now observe that  if we choose another place  $w'|v$ of $\K'$, then $w'=hw$ for some $h\in \Gal(\K'/\K)$.
Let  $\chi_{w'}:=\chi_{hw}:=\chi_w(h^{-1}(.))$; this is a non-trivial character of $\left( A_{w'}/W_{w'} \right)^{\lor}$.

\begin{lemm} Set $g_w:=\varphi_w(\chi_w)$ and $g_{w'}:=\varphi_{w'}(\chi_{w'})$.  Then  $\langle g_w\rangle=\langle g_{w'}\rangle$.
\end{lemm}

\begin{proof} This is a consequence of Kummer theory where we have  $g_{w'}=g_w^a$ for some $a\in \FF_p^\times$  (see, for example, \cite[Chapter I, \S6, Theorem 6.2]{Gras-livre}).\end{proof}

Thus, all the subsequent results do not depend on the choice of $w|v$.

\medskip

We will now describe $\varphi_v$ more precisely.

\medskip

$(i)$ This is the most important case. Let $v\in \Sigma'$. Recall that $U_v/W_v \simeq \Z/p$, hence $U_v^p \subset W_v$. There exists a non-trivial element $\chi_v $ of $\left(U_v/W_v\right)^{\lor}$ such that $$\langle \chi_v \rangle = \left(U_v/W_v\right)^{\lor} \hookrightarrow \left(U_v/U_v^p\right)^{\lor}\cdot$$

Then $\varphi_v(\chi_v)$ is an element $g_v:=g_{\chi_v}$ of $D_v$, defined by $$g_v(\sqrt[p]{\varepsilon})=\chi_v(i_v(\varepsilon))\cdot \sqrt[p]{\varepsilon},$$
for all $\varepsilon \in E_T^S$.

Let $Pl_{\K,p}=\{ v\in Pl_\K, v|p\}$ be the set of $p$-adic places of  $\K$.

Observe that if  $v \notin Pl_{\K,p}\cup S_0$, then $v$ is unramified in $\F_T^S/\K'$, and $U_v^p=W_v$.  
In particular,   $D_v$ is a cylic group generated by the Frobenius $\sigma_v$ at $v$. Thus $$\varphi_v :\langle \chi_v \rangle \twoheadrightarrow \langle \sigma_v \rangle\cdot$$
By taking a suitable power of $\chi_v$, we obtain $\varphi_v(\chi_v)=\sigma_v$.

\medskip

$(ii)$ Let $v \in S_0 \backslash S_0 \cap \Sigma$. Then $v$ is unramified in $\L/\K$.

First, note that if $v$ splits in $\L/\K$, then $W_v=\K^\times_v$ and thus $\varphi_v$ is the trivial map.

Now, suppose  $v$ is inert in $\L/\K$. Then $W_v=U_v\langle \pi_v^p\rangle$ and thus  $$\K_v^\times/W_v\simeq \K_v^\times/U_v\langle \pi_v^p\rangle\simeq \langle \pi_v\rangle /\langle \pi_v^p\rangle\cdot$$
Let $\chi_v$ be the generator of $\left(\langle \pi_v\rangle /\langle \pi_v^p\rangle\right)^{\lor}$ defined by   $\chi_v(\pi_v^i)=\zeta_p^i$.
Then  $g_v:=\varphi(\chi_v)$ satisfies: for all $\varepsilon\in E_T^S$, $$g_v(\sqrt[p]{\varepsilon})=\chi_v(\iota_v(\varepsilon))\cdot \sqrt[p]{\varepsilon}.$$
Thus $\chi_v(\iota_v(\varepsilon))=1$ if and only if the valuation $v(\varepsilon)$ of $\varepsilon$
is zero  modulo $p$.

\medskip

$(iii)$ Let $v \in S_0 \cap \Sigma$. This is analogous to $(i)$, noting that $A_v=\K_v^\times$.

\medskip

$(iv)$ 
Here $p=2$ and $v$ is a real place in $S$.

As in $(iii)$,  if $v$ is splitting in $\L/\K$, then $W_v=\K^\times_v$ and  $\varphi_v$ is the trivial map.  Otherwise for  $\varepsilon\in E_T^S$
$$g_v(\sqrt{\varepsilon})=\sign(\iota_v(\varepsilon))\cdot \sqrt{\varepsilon},$$
where $\sign(\iota_v(\varepsilon))$ is the sign of the embedding   $\iota_v(\varepsilon)$ of $\varepsilon$ in $\K_v$.


\subsection{A restriction}





Let $T=\lbrace v_1,\ldots ,v_t \rbrace$, and for $i=1, \cdots, t$, let $\sigma_{v_i}$ be the  Frobenius at $v_i$ in $\Gamma^S$; set  $\H_T:=\langle \sigma_{v} , \, v\in T\rangle$.

\begin{prop}\label{proposition_reduction_T} 
Suppose that the set $\lbrace \sigma_{v_1},\ldots ,\sigma_{v_t}  \rbrace$ forms a  $\FF_p$-free family in $\Gamma^S$. Then 
$$\Gamma_T^S:=\Gal(\F^S_T/\K') \simeq \Gamma^S/\H_T\cdot$$
\end{prop}

\begin{proof}
Let's give a proof by induction on the cardinality of $T$. Recall that for $v\in T$, one has $N_v=1$ mod $p$.
    
$\bullet$ Suppose  $T=\lbrace v \rbrace$. Let $E^S_{\lbrace v \rbrace}=\lbrace \varepsilon \in E^S,\,\,\varepsilon\equiv 1\, (v) \rbrace$. Define $\E_v^S=\{\varepsilon \in E^S,\,\,\varepsilon\in (\K_v^\times)^p \rbrace$. By Hensel's lemma, we have $ E^S_{\lbrace v \rbrace} \subset  \E_v^S$. Moreover,  $E^S/E^S_{\{v\}} \hookrightarrow \FF_v^\times$,
where $\FF_v$ is the residue field at $v$.  Thus, $E^S/E^S_{\{v\}}$ is cyclic of order dividing  $p$.
Since $ E^S_{\lbrace v \rbrace} \subset  \E_v^S$, it follows
\begin{eqnarray} \label{se0} \Z/p\Z  \twoheadrightarrow\dfrac{E^S(\K^\times)^p}{E^S_{\lbrace v \rbrace}(\K^\times)^p} \twoheadrightarrow \dfrac{E^S(\K^\times)^p}{\E^S_v(\K^\times)^p}\cdot\end{eqnarray}

By Lemma \ref{lemm_dec}, we have:
$$\left(\dfrac{E^S(\K^\times)^p}{\E^S_ v(\K^\times)^p}\right)^{\lor} = \langle \sigma_v\rangle \subset \Gamma^S\cdot$$
Now, since $\sigma_v\neq 0$ by assumption, it follows that $\dfrac{E^S(\K^\times)^p}{\E^S_v(\K^\times)^p} \simeq \Z/p\Z$. Thus, from ($\ref{se0}$) we have    $\dfrac{E^S(\K^\times)^p}{E^S_{\{v\}}(\K^\times)^p} =   \langle \sigma_v\rangle^{\lor}$,
or equivalently $\Gal(\F^S/\F_T^S)=\langle \sigma_v\rangle$.
This concludes this case. 

\medskip

$\bullet$ Let's suppose $T=T_0\cup \lbrace v \rbrace$, and that  the proposition is true for $T_0$. Define $\E^S_{T_0,v}=\lbrace \varepsilon \in E^S,\,\, \varepsilon\equiv 1\,(v'),\,\, v'\in T_0 , \,\,\varepsilon\in U_v^p \rbrace$. By Hensel's lemma, we have $E^S_{T}\subset\E^S_{T_0,v}$. Similary as before,  $\dfrac{E^S_{T_0}}{E^S_{T}}\hookrightarrow \FF_v^\times $, implying  that $\dfrac{E^S_{T_0}}{E^S_{T}(E^S_{T_0})^p }$ is a cyclic group of order dividing  $p$, so we have 
\begin{eqnarray}\label{se1} \Z/p\Z \twoheadrightarrow \dfrac{E^S_{T_0}}{E^S_{T}(E^S_{T_0})^p} \twoheadrightarrow \dfrac{E^S_{T_0}(\K^\times)^p}{E^S_{T}(\K^\times)^p}\twoheadrightarrow \dfrac{E^S_{T_0}(\K^\times)^p}{\E^S_{T_0,v}(\K^\times)^p}\cdot\end{eqnarray}
Then we define
$$\left(\dfrac{E^S_{T_0}(\K^\times)^p}{\E^S_{T_0,v}(\K^\times)^p}\right)^{\lor}=\langle \overline{\sigma}_v \rangle,$$ where $\overline{\sigma}_v$  is the restriction of the Frobenius  $\sigma_v \in \Gamma^S$ to $\F_{T_0}^S$. By the induction hypothesis, $$\left( \dfrac{E^S(\K^\times)^p}{E^S_{T_0}(\K^\times)^p}\right)^{\lor}=\langle \sigma_{v'} , \, v'\in T_0\rangle \cdot$$
But $\overline{\sigma}_v=1$ would imply  $\sigma_v \in \langle \sigma_{v'} , \, v'\in T_0\rangle $ which contradicts the assumption. Therefore the surjection in (\ref{se1}) are isomorphisms, and $\Gal(\F^S/\F_T^S)$ is generated by the Frobenius elements $\sigma_{v}$ and $\sigma_{v'}$, $v'\in T_0$. This concludes the proof.
\end{proof}

\begin{rema} 
The Galois group $\Gal\left(\F^S/\F^S_{\lbrace v\rbrace}\right)$  may not be generated by the Frobenius at $v$. Let's give an example.

Take $\K=\Q$ and $p=2$. Choose $T=\lbrace \ell \rbrace$, where $\ell \equiv 1 (\mod \ 4)$  is a prime number.
    
    Let  $S=S_\infty=\{v_\infty\}$. 
We have $E^S=\langle \pm 1\rangle$ and  $E^S_{\lbrace \ell\rbrace}=\langle 1\rangle$. Thus
$$\F^S=\Q(\sqrt{E^S})=\Q(\sqrt{-1}), \text{ and }
\F^S_{\lbrace \ell\rbrace}=\Q(\sqrt{E^S_{\lbrace \ell\rbrace}})=\Q\,.$$
However, since 
 $\ell$ is splitting in $\Q(\sqrt{-1})/\Q$, it follows that  $\sigma_\ell=1$. Consequently, 
  $\Gal(\Q(\sqrt{-1})/\Q)$ is not generated by the Frobenius at $ \ell$.
\end{rema}

When $p=2$  we can handle the archimedean places in the same way. 
Set  $\SS:=S_0\cup Pl_{\K,\infty}^{re}$. 
Observe that  $E^\SS$ is the group of $S_0$-units in the classical sense (with no sign condition). 


\begin{prop}
    Take $p=2$. Set $H^{S_\infty}=\langle  \sigma_{v};\, v\in pl_{\K,\infty}^{re}\backslash S_\infty \rangle \subset \Gamma^\SS$ and identify  $H_{S_\infty}$ with its restriction to  $\Gamma^\SS_T$.
    Then we have  $$\Gamma_T^S:=\Gal(\K(\sqrt{E^{S}_T})/\K)\simeq \Gamma^{\SS}_T/H_{S_\infty}\cdot$$
   Moreover, if the set  $\lbrace \sigma_v,\, v\in T \rbrace$ forms a lineraly independent family in $\Gamma^{\SS}$, then
    $$\Gal(\K(\sqrt{E^{S}_T})/\K)\simeq \Gamma^{\SS}/\left(H_{S_\infty}+ H_T\right)\cdot$$
\end{prop}

\begin{proof} Similar to Lemma \ref{lemm_dec}, we can show that $\K(\sqrt{E^{S}_T})$ corresponds, by Galois theory, to the subgroup   $H_{S_\infty}$ of $\Gamma^\SS_T$.

As for the second part, from Proposition \ref{proposition_reduction_T} we know that $\Gamma_T^\SS \simeq \Gamma^\SS/H_T$; thus, we conclude with the first point.
\end{proof}

 
\section{Main result}\label{section_resultatprincipal}
We keep the notations from the previous sections. In particular,  $\Sigma'=\Sigma \backslash \Sigma \cap (S\cup T)$.

\medskip

For $v\in \Sigma'\cup S$, let's consider the elements $g_v:=\varphi_v(\chi_v) \in \Gamma_T^S$ defined in  $(i)-(iv)$  of \S \ref{section_frob}. Recall that if $v$ is unramified in $\F_S^T/\K'$, then $g_v=\sigma_v$ is the  Frobenius of $v$ in $\Gamma_T^S$.

Let $\Theta^S_{\Sigma,T}$ be the following linear map: 
$$\Theta_{\Sigma,T}^S\,:\,\left(\dfrac{U_{\K,\Sigma'}^S}{ W}\right)^{\lor}\longrightarrow \Gamma_T^S. $$
defined by  $\Theta_{\Sigma,T}^S(\chi_v)=g_v$.

\begin{theo}\label{theo2}
The Artin map induces the following isomorphism:
$$\ker(\Theta^S_{\Sigma,T})\simeq \Gal(\K_{\Sigma, S,T}/\K^S_{T})^{\lor}.$$
\end{theo}

\begin{proof} First, let's observe  that
$$E^S\cap (\K^\times)^p=(E^\SS)^p,$$
where $E^\SS$ is the group of $S_0$-units in the classical sense. Then
$$E_T^S (E^\SS)^p/(E^\SS)^p \simeq E_T^S (\K^\times)^p/(\K^\times)^p\cdot$$ 
Now, consider the exact sequence obtained from Proposition~\ref{proposition1} and Remark~\ref{rema0}:
$$1\longrightarrow \nu (E^S _T(E^\SS)^p/(E^\SS)^p)\longrightarrow U^S _{\K,\Sigma'} / W \longrightarrow \Gal(\K_{\Sigma, S,T}/\K^S_{T})\longrightarrow 1\cdot$$
By Kummer duality, we have:
$$
\xymatrix{
 1 \ar[r] & \Gal(\K_{\Sigma\cup S}/\K^S_{T})^{\lor} \ar[r] & \left( U^S_{\K,\Sigma'} / W \right)^{\lor}\ar@{.>}[d] \ar[r] & \left(\nu (E^S_T(E^\SS)^p/(E^\SS)^p)\right)^{\lor}  \ar@{^{(}->}[d]\ar[r] & 1\\
  &  &\Gamma_T^S &\ar[l]_--{\Psi}^--{\simeq} \left( E^S_T(E^\SS)^p/(E^\SS)^p \right)^{\lor}
  }
$$
Now, $$\left( U^S_{\K,\Sigma'} / W \right)^{\lor}\simeq \prod_{v\in \Sigma'}(U_v/W_v)^\lor \prod_{v\in S}(U_v/W_v)^\lor .$$
Then, it suffices to observe that the induced map from  $\left( U^S_{\K,\Sigma'} / W \right)^{\lor}$ to $\Gamma_T^S$ corresponds to $\Theta_{\Sigma,T}^S$. Therefore, we finally obtain:
$$\Gal(\K_{\Sigma, S,T}/\K^S_{T})^{\lor}\simeq \ker\left(\left( U^S_{\K,\Sigma'} / W  \right)^{\lor}\stackrel{\Theta_{\Sigma,T}^S}{\longrightarrow} \Gamma_T^S \right).$$
Hence the result.
\end{proof}

Therefore, it follows that
\begin{coro}
  We have $g^S_T=\#\ker (\Theta_{\Sigma,T}^S)$.
\end{coro}

\begin{proof} This is a consequence of Theorem~\ref{theo2} and Proposition~\ref{proposition1}.
\end{proof}

It follows that if $v\in S$ splits in $\L/\K$, then the component at $v$ in $\dfrac{U_{\K,\Sigma'}^S}{ W}$ is trivial. Set  $S=S^{sp}\cup S^{ns}$, where  $S^{sp}$ is the set of places in $S$ that split in  $\L/\K$, and  $S^{ns}=S\backslash S^{sp}$. Let $s^{ns}=\#S^{ns}$ and $m:=\# \Sigma$. 

Then $\left(\dfrac{U_{\K,\Sigma'}^S}{ W}\right)^{\lor}$ is isomorphic to $(\Z/p)^{s^{ns}+m}$. 

\begin{coro} \label{corollaire rang}
    We have $g^S_T=\#\ker (\Theta_{\Sigma,T}^S)$. En particular 
    $$m+s^{ns}-r_T^S \leq log_p(g_S^T)\leq m+s^{ns},$$
     where $r_T^S$ is the $p$-rank of $E_T^S$. 
\end{coro}

\begin{proof}  Il suffices to note that  $\dim \Gamma_T^S = \dim E_T^S(\K^\times)^2/(\K^\times)^2 \leq  r_T^S$.
\end{proof}

\begin{rema}
 We have $\dim \Gamma_T^S \leq r_{S_0}$,  where $r_{S_0}=r_1+r_2+|S_0|-\delta_{\K,p}$. 
 
 When the Frobenius elements of the places
  $v \in T$ are linearly independent in $\Gamma^S$, we also have 
$\dim \Gamma_T^S=\dim \Gamma^S - |T| \leq r_{S_0}-|T|$. (See Proposition \ref{proposition_reduction_T}.)
\end{rema}

\begin{coro}[Theorem \ref{Theo1}] \label{corollaire_frobenius}
If $S^{ns}=\Sigma \cap Pl_{\K,p}=\emptyset$, let
$\{e_{v_1},\ldots , e_{v_{m}}\}$ be a basis of $(\FF_p)^{m}$ indexed by the places $v$ in $\Sigma$, and let $\Theta$ be the linear map defined by: 
$$
\begin{array}{rcll}
\Theta \, : \, (\FF_p)^{m}& \longrightarrow  & \Gamma_T^S \\
  e_{v}  & \longmapsto  & \sigma_{v} .
\end{array}$$
Then $g_S^T=\# \ker(\Theta)$.
\end{coro}

\begin{proof}
 In this case, $g_v=\sigma_v$.
\end{proof}


\section{Examples} 

\subsection{Quadratic extensions} Let's take $p=2$  and $\K=\Q$. Let $\L/\Q$ be a quadratic extension with set of ramification  $\Sigma=\{p_1,\cdots, p_m\}$. Set $\L=\Q(\sqrt{d})$, where $d\in \Z$ is square-free.

\smallskip

Let  $S_0=\lbrace \ell_1,\ldots ,\ell_s \rbrace$. We assume that $\Sigma \cap S=\emptyset$.

We denote $\ell_\infty$ as the infinite place; then $S_\infty=\lbrace \ell_\infty \rbrace$ or $S_\infty=\emptyset$. 
In the spirit of Proposition \ref{proposition_reduction_T}, we assume $T=\emptyset$.

\smallskip

Let $E^S$ be the group of $S$-units of $\Q$. We write
    $E^S=\langle\ell_0,\ldots ,\ell_s\rangle$,  with $\ell_0=-1$  or $1$ depending on whether $S_\infty=\{\ell_\infty\}$ or not.

\smallskip
    
    In this context, the governing field is written as $\F^S=\Q(\sqrt{E^S})=\Q(\sqrt{\ell_0},\ldots , \sqrt{\ell_s})$. 
    Its Galois group $\Gamma^S:=\Gal(\F^S/\Q)$ is isomorphic to $\prod_{j=0}^s \Gal(\Q(\sqrt{\ell_j})/\Q)$.
Note that $\Gal(\Q(\sqrt{\ell_0})/\Q)$ may be trivial.

\smallskip

Let's revisit the morphism $\varphi_v$ defined in Section \S \ref{section_frob} and consider its restriction to $\Q(\sqrt{\ell_j})$: its value is in $\{0,1\}$.  For what follows, the quadratic residue symbol is viewed additively, meaning it takes values in $\FF_2$.

\begin{lemm} \label{prop_quad}
The elements $g_\ell$ takes the following values:

$(a)$ For $\ell \in  \Sigma'$ and  $\ell$  odd,  $g_\ell=\sigma_\ell$ restricts to $\Q(\sqrt{\ell_j}) $ equals  $\left(\dfrac{\ell_j}{\ell}\right)$. 

$(b)$  For $\ell \in S_0^{ns}\backslash S_0^{ns}\cap \Sigma$,
  $g_\ell$  restricts  to  $\Q(\sqrt{\ell_j}) $  is trivial if and only if $ \ell \neq \ell_j$.

$(c)$  For $\ell=\ell_\infty$, 
the element  $g_{\ell_\infty}$  restricts  to  $\Q(\sqrt{\ell_j}) $  is trivial unless  $\ell_j=\ell_0=-1$ and $\L$ is imaginary.
\end{lemm}

\begin{proof} 
 $(a)$ is $(i)$ of \S \ref{section_frob}, $(b)$ is $(ii)$ and $(c)$ is $(iv)$.\end{proof}
 
 It remains to describe   $g_2$ when $2$ is ramified in $\L/\Q$. So, suppose  $2 \in \Sigma$.
 We identify $g_2$ with its restriction to $\Gal(\Q(\sqrt{\ell_i})/\Q)$. We have the following extensions 
 $$
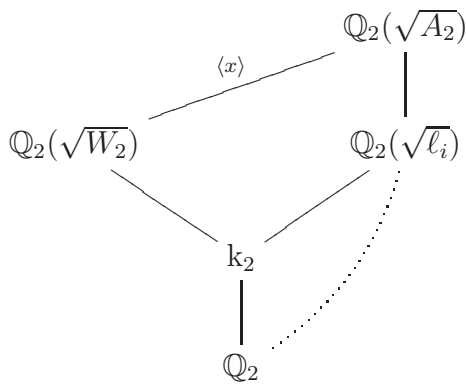
\xymatrix{&&\Q_2(\sqrt{A_2}) \\
\Q_2(\sqrt{W_2})\ar@{-}[urr]^{\langle x \rangle}&&\Q_2(\sqrt{\ell_i})\ar@{-}[u] \ar@{.}@/^1pc/[ddl]\\
&\k_2\ar@{-}[ul]\ar@{-}[ur]& \\
&\Q_2\ar@{-}[u]&
}
$$
Recall that $A_2=U_2$ (respectively $A_2=\Q_2^\times$) if $2\notin S$ (resp. $2\in S$).

The desired element $g_2$ is the image of the reduction of  $x$ in $\Gal(\Q_2(\sqrt{\ell_i})/\Q_2) \hookrightarrow \Gal(\Q(\sqrt{\ell_i})/\Q)$.
Therefore  $g_2$ (restricted) is   trivial if and only if $\ell_i \in W_2 $ modulo $(A_2)^2$.

\smallskip

Let's take for example $\ell_i=2$. 
Observe that for $2\in S$, then $2 \in \W_2$, and thus $g_2=0$. On the other hand, if $2\notin S$, then $W_2 \subset U_2$ and consequently  $g_2=1$.

\smallskip

In general, everything relies on determining $W_2$, which is the conductor at  $2$ of $\L/\K$; see   \cite[Chapiter II,  \S 1, Exercise 1.6.5]{Gras-livre} for calculations. 

For example, suppose $d\equiv -1$ modulo $8$. Then $W_2=\langle 5 \rangle$. Hence, $g_2$ restricted to $\Gal(\Q(\sqrt{\ell_i})/\Q)$ is trivial if and only if, $\ell_i \equiv$ 1  modulo $4$.

\medskip

A particularly noteworthy situation arises when we are only dealing with cases  $(a)$ and  $(b)$ of  Proposition \ref{prop_quad} and  $\L/\Q$ is unramified at  $2$. Let's detail this situation.

\medskip

 Let the canonical basis $\B:= \lbrace e_{p_1},\ldots ,e_{p_m},e_{\ell_0},  e_{\ell_1},\ldots ,e_{\ell_{s}}\rbrace$ of $\FF_2^{m+s+1}$ be indexed by the places of $\Sigma \cup \{\ell_\infty\} \cup {S_0} $; here we take $S_\infty=\{\ell_\infty\}$ and set $e_{\ell_0}:=e_{\ell_\infty}$. 

The map $ \Theta:=\Theta_\Sigma^S $ on the basis  $\B$, taking values in  $\prod_{j=0}^n \Gal(\Q(\sqrt{\ell_j})/\Q)$, is defined by   $$\Theta(e_{p_i})_{|\Q(\sqrt{\ell_j})}= \left(\dfrac{\ell_j}{p_i}\right), \ \ \Theta(e_{\ell_i})_{|\Q(\sqrt{\ell_j})}=\delta_{i,j}^* ,$$ 
where for  $0\leq i\leq s$, $\delta_{i,j}=0$ when $i\neq j$; for $j> 0$, 
$$\delta_{j,j}^*= \left\{\begin{array}{l} 1 \,\, {\rm if } \ \ell_j \ \text{\rm is inert in }\L/\Q \\
    0                          \,  \, {\rm if }  \ \ell_j  \ \text{\rm splits in }\L/\Q                                                                            \end{array}
 \right.$$ 
 and  $\delta_{0,0}^*=1$ if $\L$ is imaginary, $0$ otherwise. 
 
 The matrix $\Theta$ is then written as follows
 
 {\tiny
$$
\left(
\begin{array}{cccccccc}

\left(\dfrac{-1}{p_1}\right) & \left(\dfrac{-1}{p_2}\right) & \cdots & \left(\dfrac{-1}{p_m}\right) & \delta_{0,0}^*& 0 & \cdots & 0\\
\\

\left(\dfrac{\ell_1}{p_1}\right) & \left(\dfrac{\ell_1}{p_2}\right) & \cdots &\left(\dfrac{\ell_1}{p_m}\right) & 0 & \delta_{1,1}^* & \cdots &0\\
\\
\vdots & \vdots &  \ddots & \vdots&\vdots & \vdots & \ddots & \vdots\\
     \\
\left(\dfrac{\ell_s}{p_1}\right) & \left(\dfrac{\ell_s}{p_2}\right) & \cdots &\left(\dfrac{\ell_s}{p_m}\right) & 0 & 0 & \cdots & \delta_{s,s}^*\\
\end{array}
\right) .
$$
}

Observe that if we take $S_\infty=\emptyset$, then to obtain the matrix, simply remove the first row and the $(m+1)$th column  of the above matrix.

\medskip


\subsection{A result of existence}

In this subsection, we make explicit Corollary \ref{corollaire rang} by adapting Theorem 1.3 from \cite{Maire2}. To simplify, we only consider extensions that are splitting at infinity. Thus, we assume that $Pl_{\K,\infty}^{re} \subset S$. Note that for $p>2$, the archimedean places play no role in the  calculations.

Recall that $r_{S}=\dim E^{S}=r_1+r_2+\#S_0-1+\delta_\K$. Let $s=\# S_0$ and define
$A_S:=A_{S_0}:=\prod_{v\in S_{0}}N_v$.

\begin{theo}\label{theorem 1.1}
Let $\K$ be a number field, and let $S=S_0\cup Pl_{\K,\infty}^{re}$ be a set of places of~$\K$. Let $k,m \geq 1$,  such that $m - r_{{S}} \leq k \leq m$. Let $p$  be a prime number.  Then there exist infinitely many sets  $\Sigma$ of finite places of  $\K$ of size  $m$, such that there exists an extension  $\L/\K$ cyclic of degree  $p$, $\Sigma$-(totally) ramified, $S$-split, and with $g^{S}:=g_\emptyset^S=p^k$.

Moreover, assuming GRH,  when $m$ is fixed, such a set $\Sigma $ can be chosen such that the absolute norm $N_v$
of each of its elements $v$ is smaller than \ $p^{2r_{S} + 2} (c_1 \log\,p + c_2 \log A_S)^2$, where $c_1$ and $c_2$ are constants depending on $\K$ and $m$.
\end{theo}

\begin{proof}  The proof relies on \cite[Section 4, Proof of Theorem 1.3]{Maire2}, with two modifications due to $S$. The first modification concerns the "existence" part, and the second concerns the "quantitative" part.
    
    \smallskip
    
    $\bullet$  We need the existence of a cyclic extension  $\L/\K$ of degree $p$ with specific properties. For the first part, we utilize the Gras-Munnier theorem with splitting (see \cite[Chapter V, Section 2, Corollary 2.4.2]{Gras-livre} and \cite{GM}).

\smallskip

Let $V_S:=\{x\in \K^\times, v(x)=0, \ \forall v \notin S\}$ be the $S$-Selmer group of $\K$. 
Set $\widetilde{\F}^{S}:=\K'(\sqrt[p]{V_S})$. The exact sequence
$$1  \longrightarrow E^S/(E^S)^p \longrightarrow V_S/(\K^\times)^p \longrightarrow \Cl^S_\K[p] \longrightarrow 1$$ shows that $[\widetilde{\F}^{S}:\F^S]=O(1)$, uniformly for $p$ and $S$.

\smallskip

Let $k$ and $m$ be non-negative integers such that $m  - r_{S} \leq k \leq m$. Set $r=m-k \leq r_S$.
Let $p^l$ be the degree of $\widetilde{\F}^{S}/\F^S$. 

\smallskip

Let us take an $\FF_p$-basis $(e_i)_{i=1,\ldots , r_{S}}$ of $\Gal({\F}^{S}/\K')$ and complete it to an $\FF_p$-basis $(e_i)_{i=1,\ldots , r_{S}+l}$ of $\Gal(\widetilde{\F}^{S}/\K')$.
By Chebotarev's density theorem, let $\Sigma=\{v_1,\ldots , v_m\}$ be  places of  $\K$ prime to $p$ such that their Frobenius elements, denoted  $\sigma_{v_i}\in\Gal(\widetilde{\F}_{S}/\K')\subset\Gal(\widetilde{\F}_{S}/\K)$, satisfy:
\begin{itemize}

\item[$(a)$] $\sigma_{v_1}=-(e_1+ \cdots + e_r)$;

\item[$(b)$] for $i=2,\cdots ,r+1,\,\sigma_{v_i}=e_{i-1}$;

\item[$(c)$] for $i=r+2, \cdots , m, \,\, \sigma_{v_i}=0,$

\end{itemize}
when $r\geq 1$. When $r=0$, choose the  $v_i$'s such that  $\sigma_{v_i}=0, \,\, i=1,\cdots , m$.

\smallskip

One has $\displaystyle{\sum_{i=1}^m \sigma_{v_i}=0}$: By the Gras-Munnier theorem referenced earlier, there exists a cyclic extension $\L/\K$ of degree $p$, $\Sigma$-totally ramified and $S$-split.  Furthemore by the choice  of the  $e_i$'s and the $v_i$'s,  the map $\Theta$ of Corollary \ref{corollaire_frobenius} has rank $r$. Hence, $\Gal(\M/\K^{S})\simeq (\FF_p)^{m-r}=(\FF_p)^{k}$.

\smallskip

$\bullet$  We now need a second modification: it concerns the discriminant  $\disc(\widetilde{\F}^S)$ of $\widetilde{\F}^S$.
The introduction of  $S$ can affect   $\disc(\widetilde{\F}^S)$ of $ \widetilde{\F}^S$ compared to  $\disc(\widetilde{\F}^\emptyset)$ in two different ways.  Firstly,  the places prime to $p$ contained in $S$,  may become ramified, and for the places $w$ of $\K'$ in $S$ lying above $p$, the valuation of the local conductor $\f_w$ in a cyclic extension  $\L_w/\K_w'$ of degree $p$ may increase. Let $e_w$ be the index of ramification of $w$ in $\K'/\Q$. Observe now that 
\begin{itemize}
 \item[$-$] 
for $w|p$,  ${w(\f_w)\leq 1+ \frac{p}{p-1} e_w}$,
\item[$-$] for $w\nmid p$, $w(\f_w) \leq p-1$.
\end{itemize}
See for example \cite[Chapter III, \S 6, Remarks \& Chapter V,  \S 3, Lemma 3]{Serre}. 

If we follow the calculations from \cite[proof of Proposition 3.2]{Maire2}, we obtain:
$$|\disc(\widetilde{\F}^S)| \leq \left( |\disc(\K)| \cdot A_S \cdot p^{4[\K:\Q]}\right)^{p^{r_S+l+1}}.$$

By using the fact that $[\widetilde{\F}^{S}:\F^S]=O(1)$,
Lemma 4.4 of \cite{Maire1} provides  the result.
\end{proof}



\begin{thebibliography}{99}
\bibitem{Federer} L. J. Federer, {\em Genera theory for $S$-class groups}, Houston  J. Math. {\bf 12,  4} (1986), 497-502.
\bibitem{FLN}C. Frei, D. Loughran, R. Newton, {\em Distribution of genus numbers of abelian number fields},  Journal of the London Mathematical Society {\bf 107} (2023), 2197-2217
\bibitem{Furuta} Y. Furuta, {\em The genus field and genus number in algebraic number field}, Nagoya Math. J. {\bf 29} (1967), 281-285.
\bibitem{Gras-livre}G. Gras, {Class Field Theory}, From Theory to Practice, corr. 2nd ed., Springer Monographs in Mathematics, Springer (2005), xiii+507 pages.
\bibitem{GM}G. Gras, A. Munnier, {\em Extensions cycliques $T$-totalement ramifi\'ees}, Publ. Math. Besan\c con, 1997/98.
\bibitem{Hasse} H. Hasse, {\em Zur Geschlechtertheorie in quadratischen Zahlk\"orpern}, J. Math. Soc. Japan {\bf 3} (1951), 45-51.
\bibitem{Jau}J.-F. Jaulent, {\em L'arithm\'etique des $\ell$-extensions}, Publ. Math. Fac. Sci. Besan\c con, Fascicule 1 (1986).
\bibitem{Ko} H. Koch, A.N. Parshin, and  I.R. Safarevic Eds., Number theory II, Encycl. of Math.Sci., vol. 62, springer-verlag 1992; Algebraic Number theory, second edition 1997.
\bibitem{Leopoldt} H. W. Leopoldt, {\em Zur Geschlechtertheorie in abelschen Zahlk\"orpern}, Math. Nachr. {\bf 9} (1953), 351-362.
\bibitem{Maire1} C. Maire, {\em Finitude de tours et $p$-tours $T$-ramifi\'ees mod\'er\'ees, $S$-d\'ecompos\'ees}, J. Th\'eor. Nombres Bordeaux {\bf 8} (1996), no. 1, 47-73.
\bibitem{Maire2} C. Maire, {\em Genus theory and governing fields}, New York J. Math. {\bf 24 } (2018), 1056-1067.
\bibitem{Serre}J.-P. Serre, Corps Locaux,  Hermann, Paris, 1968.
\end{thebibliography}
\end{document}